\documentclass[12pt,twoside]{amsart}
\usepackage{amsmath, amsthm, amscd, amsfonts, amssymb, graphicx}
\usepackage{enumerate}
\usepackage[colorlinks=true,
linkcolor=blue,
urlcolor=cyan,
citecolor=red]{hyperref}
\usepackage{mathrsfs}
\newtheorem{theorem}{Theorem}[section]
\newtheorem{lemma}{Lemma}[section]
\newtheorem{remark}{Remark}[section]

\newtheorem{corollary}{Corollary}[section]

\newtheorem{proposition}{Proposition}[section]
\numberwithin{equation}{section}

\begin{document}
	
\title{More accurate numerical radius inequalities (II)}
\author{Hamid Reza Moradi and Mohammad Sababheh}

\begin{abstract}
In a recent work of the authors, we showed some general inequalities governing numerical radius inequalities using convex functions. In this article, we present results that complement the aforementioned inequalities. In particular, the new versions can be looked at as refined and generalized forms of some well known numerical radius inequalities. Among many other results, we show that
\[\left\| f\left( \frac{{{A}^{*}}A+A{{A}^{*}}}{4} \right) \right\|\le \left\| \int_{0}^{1}{f\left( \left( 1-t \right){{B}^{2}}+t{{C}^{2}} \right)dt} \right\|\le f\left( {{w}^{2}}\left( A \right) \right),\] when $A$ is a bounded linear operator on a Hilbert space having the Cartesian decomposition $A=B+iC.$ This result, for example, extends and refines a celebrated result by kittaneh.
\end{abstract}
\maketitle
\pagestyle{myheadings}
\markboth{\centerline {More accurate numerical radius inequalities (II)}}
{\centerline {H. R. Moradi \& M. Sababheh}}
\bigskip
\bigskip
\section{Introduction}
Let $\mathcal{B}(\mathcal{H})$ stand for the $C^*$ algebra of all bounded linear operators  on a complex Hilbert space $\mathcal{H}.$ Every $A\in\mathcal{B}(\mathcal{H})$ admits the Cartesian decomposition $A=B+iC,$ in which $B$ and $C$ are self adjoint operators. In this context, an operator $X\in \mathcal{B}(\mathcal{H})$ is said to be self adjoint if $X^*=X$, where $X^*$ is the adjoint operator of $X$.\\
For $A\in\mathcal{B}(\mathcal{H})$, the absolute value $|A|$ is defined by $\left| A \right|={{\left( {{A}^{*}}A \right)}^{\frac{1}{2}}}$. Notice that $|A|$ is a positive semi-definite operator, in the sense that $\left<|A|x,x\right>\geq 0$, for all $x\in\mathcal{H}.$

Among the most interesting numerical values associated with an operator $A\in\mathcal{B}(\mathcal{H})$ are the operator norm $\|A\|$ and the numerical radius $w(A)$ of $A$, defined respectively by
$$\|A\|=\sup_{\|x\|=1}\|Ax\|\;{\text{and}}\;w(A)=\sup_{\|x\|=1}\left|\left<Ax,x\right>\right|.$$

It is easy to see that $\|A\|=\sup\limits_{\|x\|=\|y\|=1}\left|\left<Ax,y\right>\right|.$ Also, it is well known that when $A$ is normal, then $\|A\|=w(A).$ 

If $A$ is not normal, then $\|A\|$ and $w(A)$ are related via the inequalities 
\begin{equation}\label{norm and numerical radius}
\frac{1}{2}\|A\|\leq w(A)\leq \|A\|.
\end{equation}

Research in this direction includes obtaining better bounds in \eqref{norm and numerical radius}. We refer the reader to \cite{10, 9, 8} for a sample of such research.

In  \cite{1}, Kittaneh proved
\begin{equation}\label{kittaneh_first_ineq}
\frac{1}{4}\left\| {{A}^{*}}A+A{{A}^{*}} \right\|\le {{w}^{2}}\left( A \right)\le \frac{1}{2}\left\| {{A}^{*}}A+A{{A}^{*}} \right\|.
\end{equation}

He also proved the following inequality in \cite{2} 
\[w\left( A \right)\le \frac{1}{2}\left\|\; \left| A \right|+\left| {{A}^{*}} \right|\; \right\|\]
and the inequality is reversed if $\frac{1}{2}$ is replaced by $\frac{1}{4}$. In fact, noting that $\|\;|A|\;\|=\|\;|A^*|\;\|$, one has
\begin{equation}\label{6}
\frac{1}{4}\left\|\; \left| A \right|+\left| {{A}^{*}} \right|\; \right\|\le \frac{1}{4}\left( \left\|\; \left| A \right|\; \right\|+\left\| \;\left| {{A}^{*}} \right|\; \right\| \right)=\frac{1}{2}\left\| A \right\|\le w\left( A \right).
\end{equation}

In this article, we target both \eqref{kittaneh_first_ineq} and \eqref{6}; where we show that these inequalities follow from a more general treatment of convex and operator convex functions. We emphasize here that the original treatment of  \eqref{kittaneh_first_ineq} and \eqref{6} did not involve any convexity approach. Therefore, we claim that our results not only are new results but also they introduce a new approach treating such inequalities.

We would like also to mention that this work can be considered as an extension of our earlier work \cite{7}; where a different set of inequalities is targeted. However, we present Theorem \ref{new_thm} and Proposition \ref{new_prop} as refinements of some results in \cite{7}, just to show how the current paper is related to \cite{7}.

For example, we show that for $A\in\mathcal{B}(\mathcal{H})$, we have the double-sided inequality:
\[\frac{1}{4}\left\| {{A}^{*}}A+A{{A}^{*}} \right\|\le {{\left\| \int_{0}^{1}{{{\left( \left( 1-t \right){{B}^{2}}+t{{C}^{2}} \right)}^{2}}dt} \right\|}^{\frac{1}{2}}}\le {{w}^{2}}\left( A \right)\] as a refinement of \eqref{kittaneh_first_ineq}. However, even this last inequality will follow as a special case of the more general inequality that when $f:\left[ 0,\infty  \right)\to \left[ 0,\infty  \right)$ is an increasing operator convex function, then
\[\left\| f\left( \frac{{{A}^{*}}A+A{{A}^{*}}}{4} \right) \right\|\le \left\| \int_{0}^{1}{f\left( \left( 1-t \right){{B}^{2}}+t{{C}^{2}} \right)dt} \right\|\le f\left( {{w}^{2}}\left( A \right) \right),\] where $A=B+iC$ is the Cartesian decomposition of $A$.\\
Another generalization is given for \eqref{6}, in a similar form. Many other generalizations and refinements of some well known results will be presented too. Further, we will present a refined version of \eqref{from 7_1}; as one can see in Theorem \ref{new_thm}.

At this stage, we pay the reader attention that in our recent work \cite{7}, a reverse-type of  the above inequality was shown as follows
\begin{equation}\label{from 7_1}
f\left( w\left( A \right) \right)\le \left\| \int_{0}^{1}{f\left( t\left| A \right|+\left( 1-t \right)\left| {{A}^{*}} \right| \right)dt} \right\|\le \frac{1}{2}\left\| f\left( \left| A \right| \right)+f\left( \left| {{A}^{*}} \right| \right) \right\|.
\end{equation}
Therefore, the current work can be thought of as an extension of some results appearing in \cite{7}.  Further refinement of this last inequality will be shown too.

Independently, we will prove a general result that implies a refinement of the well known inequality \cite{02}
$$w^p(A)\leq \left\| \left( 1-t \right){{\left| A \right|}^{p}}+t{{\left| {{A}^{*}} \right|}^{p}} \right\|,\quad 2\leq p\leq 4,$$ using certain properties of operator convex functions.

In our results, operator convex functions will be an essential assumption. Recall that a function $f:J\to\mathbb{R}$ is said to be operator convex if it is continuous and $f\left(\frac{A+B}{2}\right)\leq \frac{f(A)+f(B)}{2},$ for all self adjoint operators $A,B$ with spectra in the interval $J$. Of course, this implies that for all $0\leq t\leq 1,$ one has $f((1-t)A+tB)\leq (1-t)f(A)+tf(B).$\\
It is this context, if $f:J\to\mathbb{R}$ is a given function and $A$ is a self adjoint operator with spectrum in $J$, then $f(A)$ is defined via functional calculus. 

It can be easily seen that when $f$ is an increasing function, then $\|f(|X|)\|=f(\|X\|)$ for the self adjoint operator $X$.

It is well known that a convex function in the usual sense is not necessarily operator convex, and it is also known that the function $f(t)=t^r, r>0$ defined on $[0,\infty)$ is operator convex if and only if $r\in [1,2]$. We refer the reader to \cite{11} for related literature about operator convex functions.

It is also readily seen that an operator convex function $f:J\to\mathbb{R}$ is also convex. Therefore, such functions comply with the Hermite-Hadamard inequality 
\begin{equation}\label{HH_ineq}
f\left(\frac{a+b}{2}\right)\leq\int_{0}^{1}f((1-t)a+tb)dt\leq \frac{f(a)+f(b)}{2},\;a,b\in J.
\end{equation}
Notice that \eqref{HH_ineq} is a refinement of the convex inequality $f\left(\frac{a+b}{2}\right)\leq\frac{f(a)+f(b)}{2}.$\\
The following modified operator version of \eqref{HH_ineq} was proved in \cite{4}:
\begin{equation}\label{1}
\begin{aligned}
f\left( \frac{X+Y}{2} \right)&\le \frac{1}{2}\left[ f\left( \frac{3X+Y}{4} \right)+f\left( \frac{X+3Y}{4} \right) \right] \\ 
& \le \int_{0}^{1}{f\left( \left( 1-t \right)X+tY \right)dt} \\ 
& \le \frac{1}{2}\left[ f\left( \frac{X+Y}{2} \right)+\frac{f\left( X \right)+f\left( Y \right)}{2} \right] \\ 
& \le \frac{f\left( X \right)+f\left( Y \right)}{2}  
\end{aligned}
\end{equation}
where $f:J\to \mathbb{R}$ is an operator convex function and $X, Y$ are two self adjoint operators with spectra in $J$.

Our proofs will rely heavily on properties of convex functions and their role with inner product. Recall that if $f:J\to\mathbb{R}$ is a convex function and $A$ is a self adjoint operator with spectrum in $J$, then one has \cite{6}
\begin{equation}\label{convex_inner}
f\left( \left\langle Ax,x \right\rangle  \right)\le \left\langle f\left( A \right)x,x \right\rangle
\end{equation}
for any unit vector $x\in \mathcal{H}$.

\section{Main Results}
In this section, we present our main results, which are mainly to extend \eqref{kittaneh_first_ineq} and \eqref{6}. The main results are shown in Proposition  \ref{thm_second_main} and Theorem \ref{thm_first_main} . However, the connection with \eqref{kittaneh_first_ineq} and \eqref{6} are given in Corollaries \ref{cor_first} and \ref{cor_second}.

We would like also to mention that the assumption operator convexity will be released to scalar convexity in Proposition \ref{prop_convex}, but this will lead to a weaker form.

\subsection{Some related norm inequalities}
In \eqref{1}, replacing $X$ by $\frac{1}{2}|A|$ and $Y$ by $\frac{1}{2}|B|$, then noting that when $f$ is an increasing non-negative function, one has $\|f(|A|)\|=f(\|A\|)$, we reach the following inequalities. 
\begin{proposition}\label{thm_second_main}
Let $A,B\in \mathcal{B}\left( \mathcal{H} \right)$. If $f:\left[ 0,\infty  \right)\to \left[ 0,\infty  \right)$ is an increasing operator convex function, then
\begin{align*}
\begin{aligned}
 \left\| f\left( \frac{\left| A \right|+\left| B \right|}{4} \right) \right\|&\le   \left\| \int_{0}^{1}{f\left( \frac{\left( 1-t \right)\left| A \right|+t\left| B \right|}{2} \right)dt} \right\| \\ 
& \le \frac{1}{2}f\left( \frac{1}{2}\left\| A \right\| \right)+\frac{1}{2}f\left( \frac{1}{2}\left\| B \right\| \right).  
\end{aligned}
\end{align*}
\end{proposition}
Notice that we did not consider all inequalities appearing in \eqref{1}. Although using the other inequalities would imply further refinements, our goal in this article is to show the idea, away from getting into unnecessary computations.

We would like to emphasize that the significance of the  inequalities in Proposition \ref{thm_second_main} is not the inequalities themselves, but their applications in obtaining numerical radius and norm  inequalities  refining \eqref{kittaneh_first_ineq} and \eqref{6}.

Now letting $B=A^*$ in Proposition \ref{thm_second_main}, we obtain the following extension and refinement  of \eqref{6}. The proof follows immediately noting that $\|A\|=\|\;|A|\;\|=\|\;|A^*|\;\|.$

\begin{corollary}\label{cor_second}
Let $A\in \mathcal{B}\left( \mathcal{H} \right)$. Then for any $1\le r\le 2$,
\[\frac{1}{{{4}^{r}}}{{\left\|\; \left| A \right|+\left| {{A}^{*}} \right|\; \right\|}^{r}}\le \left\| \int_{0}^{1}{{{\left( \frac{\left( 1-t \right)\left| A \right|+t\left| {{A}^{*}} \right|}{2} \right)}^{r}}dt} \right\|\le \frac{1}{{{2}^{r}}}{{\left\| A \right\|}^{r}}.\]
In particular,
\[\frac{1}{4}\left\| \;\left| A \right|+\left| {{A}^{*}} \right| \;\right\|\le {{\left\| \int_{0}^{1}{{{\left( \frac{\left( 1-t \right)\left| A \right|+t\left| {{A}^{*}} \right|}{2} \right)}^{2}}dt} \right\|}^{\frac{1}{2}}}\le \frac{1}{2}\left\| A \right\|.\]
\end{corollary}

Notice that operator convexity of $f$ is a necessary condition, since it is so in \eqref{1}. In the next result, we show the convex version of Proposition \ref{thm_second_main}. 

\begin{theorem}\label{prop_convex}
Let $A,B\in \mathcal{B}\left( \mathcal{H} \right)$. If $f:\left[ 0,\infty  \right)\to \left[ 0,\infty  \right)$ is a  convex function, then
\begin{equation}\label{1st_concl_cov_thm}
\begin{aligned}
f\left(\left<\frac{|A|+|B|}{2}x,x\right>\right)&\leq \int_{0}^{1}f\left(\left\|((1-t)|A|+t|B|)^{1/2}x\right\|^2\right)dt\\
&\leq \frac{1}{2}\left\|f(|A|)+f(|B|)\right\|,
\end{aligned}
\end{equation}
for any unit vector $x\in\mathcal{H}$, and 
\begin{equation}\label{2ndd_concl_cov_thm}
 f\left( \left\|\frac{\left| A \right|+\left| B \right|}{2} \right\|\right) \le \underset{\left\| x \right\|=1}{\mathop{\underset{x\in \mathcal{H}}{\mathop{\sup }}\,}}\,\int_{0}^{1}{f\left( {{\left\| {{\left( \left( 1-t \right)\left| A \right|+t\left| B \right| \right)}^{\frac{1}{2}}}x \right\|}^{2}} \right)dt\le \frac{1}{2}\left\| f\left( \left| A \right| \right)+f\left( \left| B \right| \right) \right\|}.
\end{equation}
Further, if $f$ is increasing then
\begin{equation}\label{2nd_concl_cov_thm}
\left\| f\left( \frac{\left| A \right|+\left| B \right|}{2} \right) \right\|\le \underset{\left\| x \right\|=1}{\mathop{\underset{x\in \mathcal{H}}{\mathop{\sup }}\,}}\,\int_{0}^{1}{f\left( {{\left\| {{\left( \left( 1-t \right)\left| A \right|+t\left| B \right| \right)}^{\frac{1}{2}}}x \right\|}^{2}} \right)dt\le \frac{1}{2}\left\| f\left( \left| A \right| \right)+f\left( \left| B \right| \right) \right\|}.
\end{equation}

\end{theorem}
\begin{proof}
Let $x\in\mathcal{H}$ be a unit vector. Then \eqref{HH_ineq} implies

\begin{equation}\label{needed_1_conv_thm}
\begin{aligned}
f\left(\left<\frac{|A|+|B|}{2}x,x\right>\right)&=f\left(\frac{\left<|A|x,x\right>+\left<|B|x,x\right>}{2}\right)\\
&\leq \int_{0}^{1}f\left((1-t)\left<|A|x,x\right>+t\left<|B|x,x\right>\right)dt\\
&\leq \frac{f\left(\left<|A|x,x\right>\right)+f\left(\left<|A|x,x\right>\right)}{2}.
\end{aligned}
\end{equation}
Denoting $\int_{0}^{1}f\left((1-t)\left<|A|x,x\right>+t\left<|B|x,x\right>\right)dt$ by $I$, we have
\begin{equation}\label{needed_2_conv_thm}
\begin{aligned}
I&=\int_{0}^{1}f\left(\left<((1-t)|A|+t|B|)x,x\right>\right)dt\\
&=\int_{0}^{1}f\left(\left<((1-t)|A|+t|B|)^{1/2}x,((1-t)|A|+t|B|)^{1/2}x\right>\right)dt\\
&=\int_{0}^{1}f\left(\left\|((1-t)|A|+t|B|)^{1/2}x\right\|^2\right)dt.\\
\end{aligned}
\end{equation}
Further, noting convexity of $f$ and \eqref{convex_inner} we have
\begin{equation}\label{needed_3_conv_thm}
\begin{aligned}
\frac{f\left(\left<|A|x,x\right>\right)+f\left(\left<|B|x,x\right>\right)}{2}&\leq\frac{\left<f(|A|)x,x\right>+\left<f(|B|)x,x\right>}{2}\\
&=\frac{1}{2}\left<\left(f(|A|)+f(|B|)\right)x,x\right>\\
&\leq \frac{1}{2}\left\|f(|A|)+f(|B|)\right\|.
\end{aligned}
\end{equation}
Combining \eqref{needed_1_conv_thm}, \eqref{needed_2_conv_thm} and \eqref{needed_3_conv_thm}, we obtain
\begin{equation}
\begin{aligned}
f\left(\left<\frac{|A|+|B|}{2}x,x\right>\right)&\leq \int_{0}^{1}f\left(\left\|((1-t)|A|+t|B|)^{1/2}x\right\|^2\right)dt\\
&\leq \frac{1}{2}\left\|f(|A|)+f(|B|)\right\|.
\end{aligned}
\end{equation}
This proves \eqref{1st_concl_cov_thm}.\\
To prove \eqref{2ndd_concl_cov_thm}, notice that for any such $f$,
\begin{align*}
\underset{\left\| x \right\|=1}{\mathop{\underset{x\in \mathcal{H}}{\mathop{\sup }}\,}}f\left(\left<\frac{|A|+|B|}{2}x,x\right>\right)&\geq f\left(\underset{\left\| x \right\|=1}{\mathop{\underset{x\in \mathcal{H}}{\mathop{\sup }}\,}}\left<\frac{|A|+|B|}{2}x,x\right>\right)\\
&=f\left(\left\|\frac{|A|+|B|}{2}\right\|\right).
\end{align*}

To prove \eqref{2nd_concl_cov_thm}, notice that when $f$ is increasing then,
\begin{align*}
\underset{\left\| x \right\|=1}{\mathop{\underset{x\in \mathcal{H}}{\mathop{\sup }}\,}}f\left(\left<\frac{|A|+|B|}{2}x,x\right>\right)&= f\left(\underset{\left\| x \right\|=1}{\mathop{\underset{x\in \mathcal{H}}{\mathop{\sup }}\,}}\left<\frac{|A|+|B|}{2}x,x\right>    \right)\\
&=f\left(\left\|\frac{|A|+|B|}{2}\right\|\right)\\
&=\left\|f\left(\frac{|A|+|B|}{2}\right)\right\|
\end{align*}
where we have used the fact that   $\|f(|X|)\|=f(\|X\|)$ when $f$ is increasing in the last line. This together with \eqref{1st_concl_cov_thm} imply \eqref{2nd_concl_cov_thm}.
 
\end{proof}

In \cite[Corollary 2.2]{bourin}, it is shown that for the increasing convex function $f:[0,\infty)\to [0,\infty),$ one has 
$$f\left(\frac{|A|+|B|}{2}\right)\leq U\frac{f(|A|)+f(|B|)}{2}U^*,$$ for some unitary matrix $U$. Notice that this inequality implies that
$$\left\|f\left(\frac{|A|+|B|}{2}\right)\right\|\leq \frac{1}{2}\left\|f(|A|)+f(|B|)\right\|.$$
It is clear that \eqref{2nd_concl_cov_thm} provides a refinement of this inequality.

\subsection{Sharper lower bounds of the numerical radius}

In order to present our next main result (Theorem \ref{thm_first_main}), we will need the following Lemma.

\begin{lemma}\label{3}
Let $A\in \mathcal{B}\left( \mathcal{H} \right)$ have the Cartesian decomposition $A=B+iC$. Then
\[{{\left\| B \right\|}^{2}},{{\left\| C \right\|}^{2}}\le {{w}^{2}}\left( A \right).\]
\end{lemma}
\begin{proof}
If $A=B+iC$ be the Cartesian decomposition of $A$, then for any unit vector $x\in \mathcal{H}$,
\[{{\left\langle Bx,x \right\rangle }^{2}}+{{\left\langle Cx,x \right\rangle }^{2}}={{\left| \left\langle Ax,x \right\rangle  \right|}^{2}}.\]
Therefore,
\[{{\left\langle Bx,x \right\rangle }^{2}}\le {{\left| \left\langle Ax,x \right\rangle  \right|}^{2}}.\]
By taking supremum over $x\in \mathcal{H}$ with $\left\| x \right\|=1$, 
\[{{\left\| B \right\|}^{2}}={{w}^{2}}\left( B \right)\le {{w}^{2}}\left( A \right).\]
Similarly one can prove ${{\left\| C \right\|}^{2}}\le {{w}^{2}}\left( A \right)$. This completes the proof.
\end{proof}

Now Proposition \ref{thm_second_main} is utilized with the Cartesian decomposition of $A$ to obtain the following generalized form of \eqref{kittaneh_first_ineq}. This shows the significance of the proposition.
\begin{theorem}\label{thm_first_main}
Let $A\in \mathcal{B}\left( \mathcal{H} \right)$ have the Cartesian decomposition $A=B+iC$. If $f:\left[ 0,\infty  \right)\to \left[ 0,\infty  \right)$ is an increasing operator convex function, then
\begin{align*}
\left\| f\left( \frac{{{A}^{*}}A+A{{A}^{*}}}{4} \right) \right\|&\le \left\| \int_{0}^{1}{f\left( \left( 1-t \right){{B}^{2}}+t{{C}^{2}} \right)dt} \right\|\\
&\le \frac{1}{2}\left\| f\left( {{B}^{2}} \right)+f\left( {{C}^{2}} \right) \right\|\\
&\le f\left( {{w}^{2}}\left( A \right) \right).
\end{align*}

\end{theorem}

\begin{proof}
In Theorem \ref{thm_second_main}, replace $|A|$ by $2B^2$ and $|B|$ by $2C^2.$ Then direct application of Theorem \ref{thm_second_main} implies the first and second inequalities.\\
For the third inequality, notice that
\begin{align*}
\frac{1}{2}\left\| f\left( {{B}^{2}} \right)+f\left( {{C}^{2}} \right) \right\|&\leq \frac{1}{2}\left(\|f(B^2)\|+\|f(C^2)\|\right)\quad \text{(by the triangle inequality)}\\
&= \frac{1}{2}\left(f(\|B^2\|)+f(\|C^2\|)\right)\quad \text{(since $\left\| f\left( \left| X \right| \right) \right\|=f\left( \left\| X \right\| \right)$)}\\
&\leq f\left( {{w}^{2}}\left( A \right) \right)\quad \text{(by Lemma \ref{3})}.
\end{align*}
This completes the proof.
\end{proof}

Noting that the function $f(t)=t^r$ is an increasing operator convex function when $1\leq r\leq 2,$ Theorem \ref{thm_first_main} implies the following extension and refinement of the first inequality in \eqref{kittaneh_first_ineq}.
\begin{corollary}\label{cor_first}
Let $A\in \mathcal{B}\left( \mathcal{H} \right)$ with the Cartesian decomposition $A=B+iC$. Then for any $1\le r\le 2$,
\[\frac{1}{{{4}^{r}}}{{\left\| {{A}^{*}}A+A{{A}^{*}} \right\|}^{r}}\le \left\| \int_{0}^{1}{{{\left( \left( 1-t \right){{B}^{2}}+t{{C}^{2}} \right)}^{r}}dt} \right\|\le {{w}^{2r}}\left( A \right).\]
In particular,
\[\frac{1}{4}\left\| {{A}^{*}}A+A{{A}^{*}} \right\|\le {{\left\| \int_{0}^{1}{{{\left( \left( 1-t \right){{B}^{2}}+t{{C}^{2}} \right)}^{2}}dt} \right\|}^{\frac{1}{2}}}\le {{w}^{2}}\left( A \right).\]
\end{corollary}

\subsection{Sharper upper bounds of the numerical radius}
In \cite[Theorem 2.1]{7} it has been shown that
\begin{equation}\label{needed_fom_first_paper}
f\left( w\left( A \right) \right)\le \left\| \int_{0}^{1}{f\left( t\left| A \right|+\left( 1-t \right)\left| {{A}^{*}} \right| \right)dt} \right\|\le \frac{1}{2}\left\| f\left( \left| A \right| \right)+f\left( \left| {{A}^{*}} \right| \right) \right\|
\end{equation}
where $f:\left[ 0,\infty  \right)\to \left[ 0,\infty  \right)$ is an increasing operator convex function. This can be improved in the following theorem.

\begin{theorem}\label{new_thm}
	Let $A\in \mathcal{B}\left( \mathcal{H} \right)$. If $f:\left[ 0,\infty  \right)\to \left[ 0,\infty  \right)$ is an increasing convex function, then
\[f\left( w\left( A \right) \right)\le \frac{1}{2}\left\| f\left( \frac{3\left| A \right|+\left| {{A}^{*}} \right|}{4} \right)+f\left( \frac{\left| A \right|+3\left| {{A}^{*}} \right|}{4} \right) \right\|.\]
\end{theorem}
\begin{proof}
It is easy to see that if $f:J\to \mathbb{R}$ is a convex function and $a,b\in J$,
\[\begin{aligned}
 f\left( \frac{a+b}{2} \right)&=f\left( \frac{1}{2}\left( \frac{3a+b}{4}+\frac{a+3b}{4} \right) \right) \\ 
& \le \frac{1}{2}\left[ f\left( \frac{3a+b}{4} \right)+f\left( \frac{a+3b}{4} \right) \right].  
\end{aligned}\] 
Let $x\in \mathcal{H}$ be a unit vector. Replacing $a$ and $b$ by $\left\langle \left| A \right|x,x \right\rangle $ and $\left\langle \left| {{A}^{*}} \right|x,x \right\rangle $ in the above inequality, we get
\begin{equation}\label{9}
\begin{aligned}
& f\left( \left\langle \frac{\left| A \right|+\left| {{A}^{*}} \right|}{2}x,x \right\rangle  \right) \\ 
& \le \frac{1}{2}\left[ f\left( \left\langle \frac{3\left| A \right|+\left| {{A}^{*}} \right|}{4}x,x \right\rangle  \right)+f\left( \left\langle \frac{\left| A \right|+3\left| {{A}^{*}} \right|}{4}x,x \right\rangle  \right) \right] \\ 
& \le \frac{1}{2}\left[ \left\langle f\left( \frac{3\left| A \right|+\left| {{A}^{*}} \right|}{4} \right)x,x \right\rangle +\left\langle f\left( \frac{\left| A \right|+3\left| {{A}^{*}} \right|}{4} \right)x,x \right\rangle  \right] \quad \text{(by \eqref{convex_inner})}\\ 
& =\frac{1}{2}\left\langle \left\{f\left( \frac{3\left| A \right|+\left| {{A}^{*}} \right|}{4} \right)+f\left( \frac{\left| A \right|+3\left| {{A}^{*}} \right|}{4} \right)\right\}x,x \right\rangle.  \\ 
\end{aligned}
\end{equation}
On the other hand, since $f$ is increasing,
\begin{equation}\label{10}
\begin{aligned}
&	f\left( \left| \left\langle Ax,x \right\rangle  \right| \right)\\
&\le f\left( \sqrt{\left\langle \left| A \right|x,x \right\rangle \left\langle \left| {{A}^{*}} \right|x,x \right\rangle } \right) \quad \text{(by the mixed Schwarz inequality \cite[pp 75-76]{5})}\\ 
& \le f\left( \left\langle \frac{\left| A \right|+\left| {{A}^{*}} \right|}{2}x,x \right\rangle  \right) \quad \text{(by the arithmetic-geometric mean inequality)}.  
\end{aligned}
\end{equation}
Combining \eqref{9} and \eqref{10} we get 
\[f\left( \left| \left\langle Ax,x \right\rangle  \right| \right)\le \frac{1}{2}\left\langle \left\{f\left( \frac{3\left| A \right|+\left| {{A}^{*}} \right|}{4} \right)+f\left( \frac{\left| A \right|+3\left| {{A}^{*}} \right|}{4} \right)\right\}x,x \right\rangle \]
for any unit vector $x\in \mathcal{H}$. By taking supremum we have 
\[f\left( w\left( A \right) \right)\le \frac{1}{2}\left\| f\left( \frac{3\left| A \right|+\left| {{A}^{*}} \right|}{4} \right)+f\left( \frac{\left| A \right|+3\left| {{A}^{*}} \right|}{4} \right) \right\|.\]
This completes the proof of the theorem.
\end{proof}
The fact that Theorem \ref{new_thm} improves \eqref{needed_fom_first_paper} is justified in the following proposition.
\begin{proposition}\label{new_prop}
	Let $A\in \mathcal{B}\left( \mathcal{H} \right)$. If $f:\left[ 0,\infty  \right)\to \left[ 0,\infty  \right)$ is an  operator convex function, then
	\begin{equation}\label{8}
\frac{1}{2}\left\| f\left( \frac{3\left| A \right|+\left| {{A}^{*}} \right|}{4} \right)+f\left( \frac{\left| A \right|+3\left| {{A}^{*}} \right|}{4} \right) \right\|\le \left\| \int_{0}^{1}{f\left( t\left| A \right|+\left( 1-t \right)\left| {{A}^{*}} \right| \right)dt} \right\|.
\end{equation}
\end{proposition}
\begin{proof}
By the second inequality in \eqref{1},
	\[\frac{1}{2}\left[ f\left( \frac{3A+B}{4} \right)+f\left( \frac{A+3B}{4} \right) \right]\le \int_{0}^{1}{f\left( tA+\left( 1-t \right)B \right)dt}\] 
	where $A$ and $B$ are two self adjoint operators with the spectra in $J$ and $f$ is operator convex on $J$. Replacing $A$ and $B$ by $\left| A \right|$ and $\left| {{A}^{*}} \right|$, respectively, we get 
	\[ \frac{1}{2}\left[ f\left( \frac{3\left| A \right|+\left| {{A}^{*}} \right|}{4} \right)+f\left( \frac{\left| A \right|+3\left| {{A}^{*}} \right|}{4} \right) \right]\le \int_{0}^{1}{f\left( t\left| A \right|+\left( 1-t \right)\left| {{A}^{*}} \right| \right)dt}.\] 
	From this we infer \eqref{8}
\end{proof}
Letting $f(t)=t^2,$ Theorem \ref{new_thm} together with Proposition \ref{new_prop} impliy the following two-term refinement of the right inequality in \eqref{kittaneh_first_ineq}.

\begin{corollary}
Let $A\in \mathcal{B}\left( \mathcal{H} \right)$. Then
\begin{align}
 {{w}^{2}}\left( A \right)&\le \frac{1}{32}\left\| {{\left( 3\left| A \right|+\left| {{A}^{*}} \right| \right)}^{2}}+{{\left( \left| A \right|+3\left| {{A}^{*}} \right| \right)}^{2}} \right\| \label{11}\\ 
& \le \left\| \int_{0}^{1}{{{\left( t\left| A \right|+\left( 1-t \right)\left| {{A}^{*}} \right| \right)}^{2}}dt} \right\| \nonumber \\ 
& \le \frac{1}{2}\left\| {{A}^{*}}A+A{{A}^{*}} \right\| \nonumber.  
\end{align}
\end{corollary}

\begin{remark}
The constant $\frac{1}{32}$ is best possible in \eqref{11}. Actually, if we assume that \eqref{11} holds with a constant $C>0$, i.e.,
\begin{equation}\label{7}
{{w}^{2}}\left( A \right)\le C\left\| {{\left( 3\left| A \right|+\left| {{A}^{*}} \right| \right)}^{2}}+{{\left( \left| A \right|+3\left| {{A}^{*}} \right| \right)}^{2}} \right\|
\end{equation}
for any $A\in \mathcal{B}\left( \mathcal{H} \right)$, then if we choose $A$ a normal operator and use the fact that for normal operators we have $w\left( A \right)=\left\| A \right\|$, then by \eqref{7} we deduce that $\frac{1}{32}\le C$ which proves the sharpness of the constant.
\end{remark}

\subsection{Some additive refinements}
We have already seen that Corollary \ref{cor_first} refines the left side inequality in \eqref{kittaneh_first_ineq}. The refinement this corollary presents was based on a convexity approach and the refining term contains an operator integral. In the next result, we use a different approach to present a new refinement of the first inequality in \eqref{kittaneh_first_ineq}. The main tool will be the basic inequality

\begin{equation}\label{sum_diff_pos}
{\left( \frac{X+Y}{2} \right)}^{2}\leq {{\left( \frac{X+Y}{2} \right)}^{2}}+{{\left( \frac{\left| X-Y \right|}{2} \right)}^{2}}= \frac{{{X}^{2}}+{{Y}^{2}}}{2}
\end{equation}
valid for the self adjoint operators $X$ and $Y$.

\begin{theorem}
Let $A\in\mathcal{B}(\mathcal{H})$. Then
\[\frac{1}{4}\left\| {{A}^{*}}A+A{{A}^{*}} \right\|\le \frac{1}{4}{{\left\| {{\left( {{A}^{*}}A+A{{A}^{*}} \right)}^{2}}+{{\left| {{A}^{2}}+{{\left( {{A}^{*}} \right)}^{2}} \right|}^{2}} \right\|}^{\frac{1}{2}}}\le {{w}^{2}}\left( A \right).\]
\end{theorem}

\begin{proof}
Let $A=B+iC$ be the Cartesian decomposition of $A$. Then 
\begin{equation}\label{needed_last_1}
{{B}^{2}}+{{C}^{2}}=\frac{{{A}^{*}}A+A{{A}^{*}}}{2}\text{ and }{{B}^{2}}-{{C}^{2}}=\frac{{{A}^{2}}+{{\left( {{A}^{*}} \right)}^{2}}}{2}.
\end{equation}
Replacing $X$ and $Y$ by ${{B}^{2}}$ and ${{C}^{2}}$ respectively in \eqref{sum_diff_pos}, we get 	
\[\left( \frac{{{B}^{2}}+{{C}^{2}}}{2} \right)^2\leq {{\left( \frac{{{B}^{2}}+{{C}^{2}}}{2} \right)}^{2}}+{{\left( \frac{\left| {{B}^{2}}-{{C}^{2}} \right|}{2} \right)}^{2}}= \frac{{{B}^{4}}+{{C}^{4}}}{2}.\]
Consequently,
\[\left\|  \left( \frac{{{B}^{2}}+{{C}^{2}}}{2} \right)\right\|^{2}\leq \left\| {{\left( \frac{{{B}^{2}}+{{C}^{2}}}{2} \right)}^{2}}+{{\left( \frac{\left| {{B}^{2}}-{{C}^{2}} \right|}{2} \right)}^{2}} \right\|= \left\| \frac{{{B}^{4}}+{{C}^{4}}}{2} \right\|.\]
Now, if $A\in \mathcal{B}\left( \mathcal{H} \right)$ have the Cartesian decomposition $A=B+iC$, \eqref{needed_last_1} implies 
\begin{align*}
\frac{1}{16}\left\| {{A}^{*}}A+A{{A}^{*}} \right\|^2&\le \frac{1}{16}{{\left\| {{\left( {{A}^{*}}A+A{{A}^{*}} \right)}^{2}}+{{\left| {{A}^{2}}+{{\left( {{A}^{*}} \right)}^{2}} \right|}^{2}} \right\|}}\\
&= \left\| \frac{{{B}^{4}}+{{C}^{4}}}{2} \right\|\\
&\leq\frac{\|B\|^4+\|C\|^4}{2}\\
&\le {{w}^{4}}\left( A \right),
\end{align*}
where we have used Lemma \ref{3} to obtain the last inequality. This completes the proof.

\end{proof}

Our last result in this approach will be extending the inequality 
\begin{equation}\label{had_kitt_ineq}
w^p(A)\leq \left\| \left( 1-t \right){{\left| A \right|}^{p}}+t{{\left| {{A}^{*}} \right|}^{p}} \right\|,\quad 2\leq p\leq 4
\end{equation}
which was shown in \cite[Theorem 2]{02}. The approach we use here is again a convexity approach; which means that we present the main result in terms of convex or operator convex functions, then we deduce the desired refinement as a special case.
For this result, we will need the following lemma.
\begin{lemma}\label{lemma_sab}(\cite[Lemma 3.12]{01})
Let $f:\mathbb{R}\to\mathbb{R}$ be operator convex, $A,B$ be two Hermitian matrices in $\mathbb{M}_n$  and let $0\leq t \leq 1$. Then
\begin{eqnarray*}
f\left((1-t)A+t B\right)&+&2r\left(f(A)\nabla f(B)-f(A\nabla B)\right)\leq (1-t)f(A)+t f(B),
\end{eqnarray*}
where $r=\min\{t,1-t\}$ and $A\nabla B=\frac{A+B}{2}.$
\end{lemma}
Now we prove our last result.
\begin{theorem} Let $A\in\mathcal{B}(\mathcal{H})$ and let $f:[0,\infty)\to [0,\infty)$ be an increasing operator convex function. Then
\begin{equation}\label{wanted_last_1}
\begin{aligned}
& f\left( {{w}^{2}}\left( A \right) \right) \\ 
& \le \left\| \left( 1-t \right)f\left( {{\left| A \right|}^{2}} \right)+tf\left( {{\left| {{A}^{*}} \right|}^{2}} \right)-2r\left( \frac{f\left( {{\left| A \right|}^{2}} \right)+f\left( {{\left| {{A}^{*}} \right|}^{2}} \right)}{2}-f\left( \frac{{{\left| A \right|}^{2}}+{{\left| {{A}^{*}} \right|}^{2}}}{2} \right) \right) \right\| \\ 
& \le \left\| \left( 1-t \right)f\left( {{\left| A \right|}^{2}} \right)+tf\left( {{\left| {{A}^{*}} \right|}^{2}} \right) \right\|.
\end{aligned}
\end{equation}
In particular,
\begin{equation}\label{wanted_last_2}
\begin{aligned}
 {{w}^{p}}\left( A \right)&\le \left\| \left( 1-t \right){{\left| A \right|}^{p}}+t{{\left| {{A}^{*}} \right|}^{p}}-2r\left( \frac{{{\left| A \right|}^{p}}+{{\left| {{A}^{*}} \right|}^{p}}}{2}-{{\left( \frac{{{\left| A \right|}^{2}}+{{\left| {{A}^{*}} \right|}^{2}}}{2} \right)}^{\frac{p}{2}}} \right) \right\| \\ 
& \le \left\| \left( 1-t \right){{\left| A \right|}^{p}}+t{{\left| {{A}^{*}} \right|}^{p}} \right\|  
\end{aligned}
\end{equation}
for $2\le p\le 4$.
\end{theorem}
\begin{proof}
Lemma \ref{lemma_sab} implies
\[\begin{aligned}
& f\left( \left( 1-t \right){{\left| A \right|}^{2}}+t{{\left| {{A}^{*}} \right|}^{2}} \right) \\ 
& \le \left( 1-t \right)f\left( {{\left| A \right|}^{2}} \right)+tf\left( {{\left| {{A}^{*}} \right|}^{2}} \right)-2r\left( \frac{f\left( {{\left| A \right|}^{2}} \right)+f\left( {{\left| {{A}^{*}} \right|}^{2}} \right)}{2}-f\left( \frac{{{\left| A \right|}^{2}}+{{\left| {{A}^{*}} \right|}^{2}}}{2} \right) \right) \\ 
& \le \left( 1-t \right)f\left( {{\left| A \right|}^{2}} \right)+tf\left( {{\left| {{A}^{*}} \right|}^{2}} \right).
\end{aligned}\]
Therefore,
\[\begin{aligned}
& \left\| f\left( \left( 1-t \right){{\left| A \right|}^{2}}+t{{\left| {{A}^{*}} \right|}^{2}} \right) \right\| \\ 
& \le \left\| \left( 1-t \right)f\left( {{\left| A \right|}^{2}} \right)+tf\left( {{\left| {{A}^{*}} \right|}^{2}} \right)-2r\left( \frac{f\left( {{\left| A \right|}^{2}} \right)+f\left( {{\left| {{A}^{*}} \right|}^{2}} \right)}{2}-f\left( \frac{{{\left| A \right|}^{2}}+{{\left| {{A}^{*}} \right|}^{2}}}{2} \right) \right) \right\| \\ 
& \le \left\| \left( 1-t \right)f\left( {{\left| A \right|}^{2}} \right)+tf\left( {{\left| {{A}^{*}} \right|}^{2}} \right) \right\|.
\end{aligned}\]
On the other hand,
\[\begin{aligned}
 f\left( {{\left| \left\langle Ax,x \right\rangle  \right|}^{2}} \right)&\le f\left( \left\langle {{\left| A \right|}^{2\left( 1-t \right)}}x,x \right\rangle \left\langle {{\left| {{A}^{*}} \right|}^{2t}}x,x \right\rangle  \right) \\ 
& \le f\left( {{\left\langle {{\left| A \right|}^{2}}x,x \right\rangle }^{1-t}}{{\left\langle {{\left| {{A}^{*}} \right|}^{2}}x,x \right\rangle }^{t}} \right) \\ 
& \le f\left( \left( 1-t \right)\left\langle {{\left| A \right|}^{2}}x,x \right\rangle +t\left\langle {{\left| {{A}^{*}} \right|}^{2}}x,x \right\rangle  \right) \\ 
& =f\left( \left\langle \left( 1-t \right){{\left| A \right|}^{2}}+t{{\left| {{A}^{*}} \right|}^{2}}x,x \right\rangle  \right)  
\end{aligned}\]
for any unit vector $x\in \mathcal{H}$. This implies that
\[f\left( {{w}^{2}}\left( A \right) \right)\le \left\| f\left( \left( 1-t \right){{\left| A \right|}^{2}}+t{{\left| {{A}^{*}} \right|}^{2}} \right) \right\|.\]
Consequently,
\[\begin{aligned}
& f\left( {{w}^{2}}\left( A \right) \right) \\ 
& \le \left\| \left( 1-t \right)f\left( {{\left| A \right|}^{2}} \right)+tf\left( {{\left| {{A}^{*}} \right|}^{2}} \right)-2r\left( \frac{f\left( {{\left| A \right|}^{2}} \right)+f\left( {{\left| {{A}^{*}} \right|}^{2}} \right)}{2}-f\left( \frac{{{\left| A \right|}^{2}}+{{\left| {{A}^{*}} \right|}^{2}}}{2} \right) \right) \right\| \\ 
& \le \left\| \left( 1-t \right)f\left( {{\left| A \right|}^{2}} \right)+tf\left( {{\left| {{A}^{*}} \right|}^{2}} \right) \right\|.
\end{aligned}\]
This proves \eqref{wanted_last_1}. For \eqref{wanted_last_2}, let $f(t)=t^r, 1\leq r\leq 2,$ and apply \eqref{wanted_last_1}.
\end{proof}

{\tiny \vskip 0.3 true cm }

{\tiny (H. R. Moradi) Department of Mathematics, Payame Noor University (PNU), P.O. Box 19395-4697, Tehran, Iran.}

{\tiny \textit{E-mail address:} hrmoradi@mshdiau.ac.ir }

{\tiny \vskip 0.3 true cm }

{\tiny (M. Sababheh) Department of Basic Sciences, Princess Sumaya University For Technology, Al Jubaiha, Amman 11941, Jordan.}

{\tiny \textit{E-mail address:} sababheh@psut.edu.jo}
\end{document}